\documentclass[11pt]{article}
\usepackage{amsmath,amssymb,amsthm,comment}
\usepackage{amsfonts}
\usepackage{color}
\usepackage{epsfig}
\newtheorem{theo}{Theorem}[section]
\newtheorem{lem}[theo]{Lemma}

\newtheorem{rem}[theo]{Remark}
\newtheorem{prop}[theo]{Proposition}

\makeatletter
    
    \@addtoreset{equation}{section}
  \makeatother
\newcommand{\nn}{\nonumber}
\newcommand{\io}{\int_0^1}

\newcommand{\dx}{\partial_x}
\newcommand{\dt}{\partial_t}
\begin{document}

\title{Boundedness of solutions to the critical fully parabolic quasilinear one-dimensional Keller-Segel system}

\author{Bartosz Bieganowski\\
{\small Nicolaus Copernicus University, Faculty of Mathematics and Computer Science, ul. Chopina 12/18,}\\
{\small  87-100 Toru\'n, Poland, e-mail: bartoszb@mat.umk.pl,}\\
Tomasz Cie\'slak  \\
{\small Institute of Mathematics, Polish Academy of Sciences, \'Sniadeckich 8, 00-656 Warsaw, Poland,}\\
{\small e-mail: cieslak@impan.pl,}\\
Kentarou Fujie\\
{\small Department of Mathematics, Tokyo University of Science, Tokyo, 162-8601, Japan,}\\
{\small e-mail: fujie@rs.tus.ac.jp}\\
Takasi Senba\\
{\small Faculty of Science, Fukuoka University, Fukuoka, 814-0180, Japan,}\\
{\small e-mail: senba@fukuoka-u.ac.jp}}
\maketitle
\date{}

\begin{abstract}
  \noindent
  In this paper we consider a one-dimensional fully parabolic quasilinear Keller-Segel system with critical
  nonlinear diffusion. We show uniform-in-time boundedness of solutions, which means, that unlike in higher dimensions, there is no critical mass phenomenon in the case of critical diffusion. 
  To this end we utilize estimates from a well-known Lyapunov functional and a recently introduced new Lyapunov-like functional in \cite{CF1}. 

\end{abstract}
\section{Introduction}\label{section1}

In the present paper we address the problem of boundedness of solutions to the one-dimensional fully parabolic quasilinear critical Keller-Segel system. It has been known that its parabolic-elliptic J\"{a}ger-Luckhaus type version, unlike in higher dimensions, does not admit the critical mass phenomenon (\cite{cl3}), i.e. in the critical diffusion case $a(u)=(1+u)^{-1}$, independently on the magnitude of initial mass, solutions stay bounded. Such a diffusion is a particular case of nonlinear diffusion $a(u)=(1+u)^{1-2/n}$ which leads to the critical mass distinguishing between global bounded solutions and blowing up ones in $n$-dimensional domains, see \cite{CF1} and references therein for the parabolic-elliptic case, in the fully parabolic case one should consult \cite{MW,TCCS,NSY,tao_winkler,lau_miz}. Let us mention here that in the parabolic-elliptic case the lack of critical mass in the critical one-dimensional non-local diffusion is shown in \cite{bg}.

The parabolic-elliptic surprising result in \cite{cl3} is based on peculiar change of variables reducing the problem to a single parabolic equation which is very sensitive to the structure of the problem, namely it cannot be extended beyond the parabolic-elliptic formulation, it apparently works only in the J\"{a}ger-Luckhaus type form of the Keller-Segel system. 

The question present in the literature was what is the situation in the fully parabolic case. It was known that the subcritical diffusion leads to global existence of solutions independently of the magnitude of initial mass, see \cite{BCMR}, while finite-time blowups occur (at least for large masses) in the supercritical case, see \cite{TCPhL}. Recently, second and third authors in \cite{CF1} have shown that in the fully parabolic critical diffusion case solutions exist globally for any initial mass. To this end a new Lyapunov-like functional was introduced. It was up to now not clear whether the critical mass phenomenon appears in the one-dimensional fully parabolic system in a slightly weaker form, namely whether there exists a threshold value of mass distinguishing between global-in-time bounded solutions and global-in-time unbounded ones. Such a situation is present in the critical two-dimensional quasilinear Keller-Segel type volume filling system, see \cite[Theorem 4.2]{TCCS}.

In the present paper we show that in the critical fully parabolic one-dimensional Keller-Segel system any solution stays bounded. This answers in a negative way a question concerning existence of a critical mass distinguishing between bounded and unbounded (though still global-in-time) solutions. Actually, our result is slightly more general and extends boundedness result to a wider class of nonlinear diffusions and showing the boundedness in the case where global existence was already shown in \cite{CF1}. It extends the parabolic-elliptic result in \cite{TCPhL1}. Finally, we also establish a similar result in a parabolic-elliptic classical version (global existence was shown already in \cite{CF2}). 

Let us formulate our result in a precise way.           

We consider the following parabolic-parabolic Keller-Segel system of equations
\begin{eqnarray}
\label{eq}
\left\{ \begin{array}{ll}
\partial_t u = \partial_x \left( a(u) \partial_x u - u \partial_x v \right), \ & \ \mathrm{in} \ (0,\infty) \times (0,1), \\
\partial_t v = \partial_{x}^2 v -  v + u, \  & \ \mathrm{in} \ (0,\infty) \times (0,1), \\
u(0,x) = u_0 (x), \ v(0,x) = v_0 (x), \ & \ \mathrm{for} \ x \in (0,1).
\end{array} \right.
\end{eqnarray}
We impose the following boundary conditions
$$
\partial_x u = \partial_x v = 0 \quad \quad \mathrm{on} \ (0,\infty) \times \partial (0,1) = (0,\infty) \times \{0, 1\}.
$$

We assume that the positive nonlinearity $a : [0,\infty) \rightarrow (0,\infty)$ satisfies the following conditions:
\begin{enumerate}
\item[(A0)] $a \in C[0,\infty) \cap C^2(0,\infty)$,
\item[(A1)] there is $\alpha > 0$ such that
$$
sa(s) \leq \alpha
$$
for any $s \geq 0$,
\item[(A2)] $\int_1^\infty a(s) \, ds = \infty$, i.e. $a \not\in L^1 (1,\infty)$.
\end{enumerate}

In what follows we will denote
\begin{equation}\label{Lambda}
\Lambda(r):=\int_1^r a(s)ds.
\end{equation}



\begin{rem}\label{Uwaga}
The function $a(s) = 1/(1+s)$ satisfies assumptions (A0)-(A2). 
\end{rem}

Our goal is to show the boundedness of the classical solution $(u,v)$. We will show the following theorem.

\begin{theo}\label{ThMain}
Assume that (A0)-(A2) hold, $(u_0, v_0) \in (W^{1,\infty}(0,1))^2$ is the pair of nonnegative initial data. The solution $(u,v)$ of \eqref{eq} is bounded.
\end{theo}
\begin{rem}
In view of Remark \ref{Uwaga} we notice that Theorem \ref{ThMain} guarantees that the main announced result holds. Namely 
solution $(u,v)$ of \eqref{eq} with $a(u)=\frac{1}{1+u}$ is bounded uniformly in time.
\end{rem}

Let us briefly describe the structure of the paper and comment on the main new estimate 
which gives enough information to conclude boundedness of solutions. 
In the second section we recall well-known existence results, the classical Lyapunov functional as well 
as the Lyapunov-like functional introduced recently in \cite{CF1}. 
Third section contains a time-independent estimate of the Lyapunov-like functional. 
Though in \cite{CF1} global-in-time existence of solutions is proven, the question of uniform boundedness of solutions has been left unsolved. It is obtained in the present paper by matching the new estimates of the classical Lyapunov functional  with the identity stating the Lyapunov-like functional. It seems to be the main novelty of this paper.  
In the fourth section, we improve the regularity of solution and arrive at $L^\infty$ bound. It is done in a completely different manner than in \cite{CF1}. Much simpler and self-contained one. It seems another advantage of the present approach over the one presented in \cite{CF1}. Finally, in the last section we formulate and discuss the analogous results in an easier parabolic-elliptic case.

\section{Preliminaries}

In the present section we collect some known results concerning the studied problem.

The following global existence result is known, see \cite[Appendix]{CF1}.

\begin{theo}\label{gloex}
Suppose that (A0), (A1), (A2) are satisfied, $(u_0, v_0) \in (W^{1,\infty}(0,1))^2$ is the pair of nonnegative initial data. Then the problem \eqref{eq} has a unique classical positive solution, which exists globally in time. Moreover the solution $(u,v)$ satisfies the mass conservation
$$
M := \int_0^1 u_0 (x) \, dx = \int_0^1 u(t,x) \, dx \quad \mbox{for all} \ t \in (0,\infty).
$$
\end{theo}
Since the obtained solution is regular, applying the strong maximum principle and the Hopf principle to the upper equation of \eqref{eq}, due to
the zero Neumann boundary condition, we know that, unless $u_0(x)\equiv 0$, for $0<t<\infty$ the following holds:
\begin{equation}\label{odciecie}
u(t,x)>0 \;\;\mbox{for all}\;\;x\in[0,1]. 
\end{equation}

By the semigroup theory we can easily get the following regularity estimate.

\begin{lem}\label{oszv}
There is a constant $C > 0$, which depends on $M$, $p \in [1,\infty]$ and $\|v_0\|_{L^p (0,1)}$ such that
$$
\sup_{t \geq 0} \| v (t,\cdot) \|_{L^p (0,1)} \leq C.
$$
\end{lem}
In particular, we know that $v$ is bounded, i.e. $v \in L^\infty ([0,\infty), L^\infty (0,1))$.

Let us also recall Lyapunov functional associated with the problem, a new Lyapunov-like functional introduced in \cite{CF1} as well as some estimates appearing as a consequence. 
The classical Lyapunov functional associated with \eqref{eq} is
$$
L(u,v) := \int_0^1 b(u) \, dx - \int_0^1 uv \, dx + \frac{1}{2} \|v\|^2_{H^1(0,1)},
$$
where $b \in C^2(0,\infty)$ satisfies $b''(r)= \frac{a(r)}{r}$ for $r > 0$ and $b(1)=b'(1)=0$.
It satisfies, see \cite{TCPhL},
$$
\frac{d}{dt} L(u,v) = -\int_0^1 |\partial_t v|^2 \, dx - \int_0^1 u \left| \frac{a(u)}{u} \partial_x u - \partial_x v \right|^2 \, dx.
$$
 Moreover, boundedness of $v$ and mass conservation yield that $L$ is bounded from below:
\begin{equation}\label{lowerbound_L}
|L(u(t),v(t))| \leq C \quad \mbox{for all }t>0
\end{equation} 
with some $C$, hence
\begin{equation}\label{v-t}
\int_0^\infty \int_0^1 |\partial_t v|^2 \, dx \, dt < \infty. 
\end{equation}
Again, for the proofs we refer to \cite{TCPhL, CF1}.

Let us recall the crucial new Lyapunov-like functional which was found in \cite{CF1} and
yields an estimate sufficient to prove global existence result stated in Theorem \ref{gloex}.
Moreover, we are going to utilize it when proving boundedness of solutions.

For ${\cal F}$ and ${\cal D}$ defined as follows
\begin{equation}\label{F}
{\cal F} (u) := \frac{1}{2} \int_0^1 \frac{(a(u))^2}{u} | \partial_x u|^2 \, dx - \int_0^1 u \int_1^u a(r) \, dr \, dx
\end{equation}
and
$$
{\cal D} (u,v) := \int_0^1 ua(u) \left| \partial_x \left( \frac{a(u)}{u} \partial_x u \right) - \partial_x^2 v + \frac{(v+\partial_t v)}{2} \right|^2 \, dx,
$$
the following claim was proven in \cite[Lemma 3.3]{CF1}.

\begin{lem}\label{lemma_Lyap1}
Let $(u,v)$ be a solution of \eqref{eq} in $(0,T)\times (0,1)$ and let $T$ be the maximal existence time of the classical solution. 
Then the following identity holds:
\begin{equation}\label{wazne}
\frac{d}{dt} {\cal F} (u) + {\cal D}(u,v) = \int_0^1 \frac{a(u)u (v + \partial_t v)^2}{4} \, dx.
\end{equation}
\end{lem}

\section{Main estimates}

This section is devoted to the main estimate, the time-independent estimate of ${\cal F}$, for the definition see \eqref{F} before Lemma \ref{lemma_Lyap1}. 
In \cite{CF1} the authors invoked Lemma \ref{oszv} and \eqref{v-t} to obtain a bound of the right hand side of \eqref{wazne} which enabled them to arrive at the inequality $\frac{d}{dt} {\cal F} \leq C$. However, such an inequality yields a time-dependent estimate of ${\cal F}$. 
The new idea of the present paper is to add ${\cal F}$ to \eqref{wazne} which gives
$$
\frac{d}{dt} {\cal F} (u)+{\cal F} (u) + {\cal D}(u,v) = \int_0^1 \frac{a(u)u (v + \partial_t v)^2}{4} \, dx
+{\cal F} (u).
$$   
Hence we infer the bound on ${\cal F}$  
$$
{\cal F} (u(t)) \leq {\cal F} (u(t_0)) + \int_{t_0}^t e^{s-t} \int_0^1 \frac{a(u)u (v + \partial_t v)^2}{4} \, dxds
+\int_{t_0}^t e^{s-t} {\cal F} (u)\,ds,
$$
which is a time-independent one provided the time-weighted estimates of the right hand side of the above inequality hold. Although the bound of the second term on the right hand side is straightforward, it seems non-trivial to find a time-independent estimate of the last term. 
We derive the latter utilizing the classical Lyapunov functional and several new estimates listed below.

\begin{prop}\label{lem_est_Lambda}
Assume (A0) and (A1) are satisfied. 
Let $(u,v)$ be a solution of \eqref{eq}. 
Then there exists some $C>0$ such that for any $t>0$
\begin{equation*}
 \io |\Lambda(u)|^2\,dx \leq C\io (u+1)\,dx, 
\end{equation*}
where $\Lambda$ is given by \eqref{Lambda}. 
\end{prop}

\begin{proof}
First, by (A0) there exists a constant $C>0$ such that for $r \in (0,1)$,
\begin{equation*}
 \left|\int_1^r a(s) \, ds \right|^2 \leq C.
\end{equation*}
Next, by (A1), the inequality $(\log x)^2 \leq \frac{4}{e^2}x $ for $x\geq 1$, it follows for $r\geq 1$,
\begin{equation*}
\left| \int_1^r a(s)\,ds \right|^2
 \leq \left|  \int_1^r \frac{\alpha}{s}\,ds \right|^2
  = \alpha^2 (\log r)^2
  \leq \frac{4\alpha^2}{e^2}r
\end{equation*}
and then by combining the above inequalities we arrive at the desired inequality.
\end{proof}

Observe that in view of the continuous embedding $L^2 (0,1) \subset L^1 (0,1)$, the mass conservation and Proposition \ref{lem_est_Lambda} we get
\begin{equation}\label{pr2}
\int_0^1 |\Lambda(u)| \, dx = \int_0^1 \left|\int_1^u a(s) \, ds\right| \,dx \leq C (M+1),
\end{equation} 
where $C > 0$. 

\begin{prop}\label{lem_est_uLambda}
Assume (A0) and (A1) are satisfied. 
Let $(u,v)$ be a solution of \eqref{eq}. 
Then there exists some $c>0$ such that
\begin{equation*}
\int_0^1 u \int_1^u a(r) \, dr \, dx 
\leq
c M (M+1) + M^{3/2} \left\| \frac{a(u)}{\sqrt{u}} \dx u \right\|_{L^2(0,1)}. 
\end{equation*}
\end{prop}
\begin{proof}
Note that mass conservation, \eqref{pr2} and one-dimensional Sobolev embeddings yield
\begin{align*}
&\quad  \int_0^1 u \int_1^u a(r) \, dr \, dx 
\leq  M \left\| \int_1^u a(r) \, dr \right\|_{L^\infty(0,1)}
 \leq  M \left\| \int_1^u a(r) \, dr \right\|_{W^{1,1}(0,1)} \\
&=  M \left\| \int_1^u a(r) \, dr \right\|_{L^1 (0,1)} 
+ M \int_0^1 \left| \partial_x \left( \int_1^u a(r) \, dr \right) \right| \, dx 
\leq  c M (M+1) + M \int_0^1 \sqrt{u} \frac{a(u) |\partial_x u|}{\sqrt{u}} \, dx \\
&\leq c M (M+1) + M^{3/2} \left( \int_0^1 \frac{(a(u))^2}{u} |\partial_x u|^2 \, dx \right)^{1/2}
\end{align*}
with some $c>0$.
\end{proof}

\begin{lem}\label{lem_keyest}
Assume (A0) and (A1) are satisfied. 
Let $(u,v)$ be a solution of \eqref{eq}. 
Then there exists some $C=C(M, \|v_0\|_{L^2(0,1)})>0$ such that
\begin{equation*}
\io a(u) \dx u \cdot \dx v \,dx \leq 
C(M, \|v_0\|_{L^2(0,1)}) \left(\|\dt v\|_{L^2(0,1)} + \left\| \frac{a(u)}{\sqrt{u}} \dx u \right\|_{L^2(0,1)}+1  \right).
\end{equation*}
\end{lem}

\begin{proof}
By the integration by parts and the second equation of \eqref{eq}  it follows that
\begin{align}\label{ineq1}
\nn
\io a(u) \dx u \cdot \dx v \,dx 
=&  \io \dx \Lambda (u) \cdot \dx v \,dx\\
=& - \io \Lambda (u ) \dx^2 v\,dx
= - \io \Lambda (u) (\dt v +v-u)\,dx,
\end{align}
where $\Lambda$ is given by \eqref{Lambda}. 
Next, from the mass conservation and Proposition \ref{lem_est_Lambda} we see that  
\begin{align}\label{ineq2}
\nn
\left|  \io \Lambda (u) \dt v \, dx  \right| 
\leq& \|\Lambda(u)\|_{L^2(0,1)} \|\dt v\|_{L^2(0,1)}\\
\leq& \left( C\io (u+1)\,dx \right)^{\frac{1}{2}}  \|\dt v\|_{L^2(0,1)}
\leq  C \|\dt v\|_{L^2(0,1)}
\end{align}
with some $C(M)>0$. Proceeding the similar way, Lemma \ref{oszv} guarantees that 
\begin{align}\label{ineq3}
\left|  \io \Lambda (u)  v \, dx  \right| 
\leq& \left( C\io (u+1)\,dx \right)^{\frac{1}{2}}  \| v\|_{L^2(0,1)}
\leq  C(M, \|v_0\|_{L^2(0,1)}).
\end{align} 
Finally collecting \eqref{ineq1}, \eqref{ineq2}, \eqref{ineq3} and Proposition \ref{lem_est_uLambda} completes the proof.
\end{proof}

\begin{lem}\label{lem_Keyest2}
Assume (A0) and (A1) are satisfied. 
Let $(u,v)$ be a solution of \eqref{eq}. 
Then there exists some $C(u_0, v_0)>0$ such that for any $t>0$
\begin{equation*}
\int_0^t e^{s-t} \io \left( \frac{(a(u))^2}{u}|\dx u|^2 + u|\dx v|^2 \right) \, dx \, ds \leq C(u_0, v_0).
\end{equation*}
\end{lem}

\begin{proof}
The classical Lyapunov functional identity can be rewritten as 
\begin{align*}
\frac{d}{dt} L(u,v) +\int_0^1 |\partial_t v|^2 \, dx 
+ \int_0^1 \left( u \left| \frac{a(u)}{u} \partial_x u \right|^2 + u\left|\partial_x v \right|^2 \right) \, dx
=  2 \int_0^1 a(u) \partial_x u \cdot \partial_x v  \, dx.
\end{align*}
Applying Lemma \ref{lem_keyest} we see that
\begin{align*}
&\quad \frac{d}{dt} L(u,v) + \|\dt v\|_{L^2(0,1)}^2 +\left\| \frac{a(u)}{\sqrt{u}} \dx u \right\|_{L^2(0,1)}^2 
+  \io u\left|\partial_x v \right|^2\,dx  \\
&\leq  C \left(\|\dt v\|_{L^2(0,1)} + \left\| \frac{a(u)}{\sqrt{u}} \dx u \right\|_{L^2(0,1)} +1 \right)
\end{align*}
with some $C=C(M, \left\|v_0\right\|_{L^2})>0$ and then
\begin{align*}
\frac{d}{dt} L(u,v) +\frac{1}{2} \|\dt v\|_{L^2(0,1)}^2 +\frac{1}{2} \left\| \frac{a(u)}{\sqrt{u}} \dx u \right\|_{L^2(0,1)}^2 
+  \io u\left|\partial_x v \right|^2\,dx 
\leq C^\prime.
\end{align*}
With the boundedness of the functional $L$ (see \eqref{lowerbound_L}), we have that
\begin{align*}
\frac{d}{dt} L(u,v) + L(u,v)+\frac{1}{2} \|\dt v\|_{L^2(0,1)}^2 +&\frac{1}{2} \left\| \frac{a(u)}{\sqrt{u}} \dx u \right\|_{L^2(0,1)}^2 
+  \io u\left|\partial_x v \right|^2\,dx \\ 
& \leq C^\prime+ L(u,v) \leq C'',
\end{align*}
hence for any $t>0$
\begin{align*}
L(u(t),v(t))+\frac{1}{2}\int_0^t e^{s-t} \io \left( \frac{(a(u))^2}{u}|\dx u|^2 + u|\dx v|^2 \right) \, dx\, ds
\leq e^{-t}L(u_0,v_0) +C'',
\end{align*}
which completes the proof.
\end{proof}

In view of \eqref{odciecie}, we can pick up a small $t_0>0$, such that $u(t_0,x)>0$ for any $x\in[0,1]$.

We are now in a position to complete the time-independent estimate of ${\cal F}$.

\begin{lem}\label{Lem:UpperBoundOfF}
Let (A0) and (A1) be satisfied. Assume that $(u,v)$ is a classical solution of \eqref{eq}. Then there is a constant $C_1(u_0, v_0) > 0$ such that
$$
{\cal F}(u(t)) \leq {\cal F}(u(t_0)) + C_1(u_0, v_0) \quad \mbox{for all} \ t \in (t_0,\infty).
$$
\end{lem}

\begin{proof}
Observe that
\begin{align*}
\int_0^t e^{s-t}  {\cal F}(u(s))\,ds &\leq \int_0^t e^{s-t} \left| \frac{1}{2} \int_0^1 \frac{(a(u(s)))^2}{u} | \partial_x u |^2 \, dx \right| \, ds + \int_0^t e^{s-t} \left| \int_0^1 u(s) \int_{1}^{u(s)} a(r) \, dr \, dx \right| \, ds \\
&=: I_1 + I_2.
\end{align*}
By Lemma \ref{lem_Keyest2} we get
\begin{equation}\label{I1}
I_1 \leq C(u_0, v_0).
\end{equation}
In view of Proposition \ref{lem_est_uLambda} and Lemma \ref{lem_Keyest2} we have
\begin{align}\label{I2}
I_2 &\leq \int_0^t e^{s-t} \left( cM(M+1) + M^{3/2} \left\| \frac{a(u)}{\sqrt{u}} \partial_x u \right\|_{L^2 (0,1)} \right) \, ds\nonumber \\ &= cM(M+1) \int_0^t e^{s-t} \, ds + M^{3/2} \int_0^t e^{s-t} \left\| \frac{a(u)}{\sqrt{u}} \partial_x u \right\|_{L^2 (0,1)} \, ds  \nonumber \\
\nn
&\leq  cM(M+1) \left( 1 - e^{-t} \right) 
+ M^{3/2} \int_0^t e^{s-t} \left( \frac{1}{2} \left\| \frac{a(u)}{\sqrt{u}} \partial_x u \right\|_{L^2 (0,1)}^2+ \frac{1}{2} \right) \, ds \\ 
&\leq  cM(M+1) + M^{3/2} C(u_0, v_0) +\frac{1}{2} M^{3/2}. 
\end{align}
Combining \eqref{I1} and \eqref{I2} we have some $C>0$ such that
\begin{align}\label{bound_intF}
\int_0^t e^{s-t}  {\cal F}(u(s))\,ds \leq C \quad \mbox{for all }t>0.
\end{align}
Adding $\mathcal{F}$ to \eqref{wazne} we have
\begin{align*}
\frac{d}{dt} {\cal F} (u) +{\cal F} (u)+ {\cal D}(u,v) = \int_0^1 \frac{a(u)u (v + \partial_t v)^2}{4} \, dx +{\cal F} (u).
\end{align*}
By (A1) and Lemma \ref{oszv} it follows that
\begin{align*}
\frac{d}{dt} {\cal F} (u) +{\cal F} (u)+ {\cal D}(u,v) 
\leq \frac{\alpha}{2}  \int_0^1 (v^2 +|\partial_t v|^2) \, dx +{\cal F} (u)
\leq  C +\frac{\alpha}{2}   \io |\dt v|^2  +{\cal F} (u).
\end{align*}
Using \eqref{v-t} and \eqref{bound_intF} we have
\begin{align*}
&{\cal F} (u(t)) + \int_{t_0}^t e^{s-t} {\cal D}(u(s),v(s))\,ds\\
&\leq {\cal F}(u(t_0)) +  C+ \frac{\alpha}{2} \int_0^t e^{s-t} \io |\dt v(s)|^2\,ds  + \int_0^t e^{s-t} {\cal F} (u(s))\,ds\\
&\leq {\cal F}(u(t_0)) +C_1,
\end{align*}
which completes the proof.
\end{proof}

\section{Improved regularity}

Once we have the time-independent bound of ${\cal F}$, we can achieve the time-independent bounds of $u$ ending up with the $L^\infty$ estimate. All the procedure is different than the one used in \cite{CF1}. It is self-contained, seems to be smarter than the one in \cite{CF1} and one concludes the claim of Theorem \ref{ThMain} directly from the bounds obtained via Lyapunov-like functional ${\cal F}$. We do not need any bootstrap procedure to increase the integrability of $u$. Instead, we directly arrive at $L^\infty$ estimate of the solution.

\begin{lem}\label{Lem:LowerBoundOfF}
Let $(u,v)$ solve \eqref{eq}. Then there is $C_2 = C_2 (M)> 0$ such that
$$
{\cal F} (u) \geq \frac{1}{4} \int_0^1 \frac{(a(u))^2}{u} | \partial_x u|^2 \, dx - C_2.
$$
\end{lem}

\begin{proof}
Proposition \ref{lem_est_uLambda} implies that
\begin{align*}
{\cal F}(u) &\geq \frac{1}{4} \int_0^1 \frac{(a(u))^2 | \partial_x u|^2}{u} \, dx + \frac{1}{4} \int_0^1 \frac{(a(u))^2 | \partial_x u|^2}{u} \, dx - M^{3/2} \left( \int_0^1 \frac{(a(u))^2}{u} |\partial_x u|^2 \, dx \right)^{1/2} - c  M(M+1) \\
&\geq \frac{1}{4} \int_0^1 \frac{(a(u))^2 | \partial_x u|^2}{u} \, dx + \frac{1}{4} \left( \left( \int_0^1 \frac{(a(u))^2 | \partial_x u|^2}{u} \, dx \right)^{1/2} - 2 M^{3/2} \right)^2 - M^3 - c M(M+1) \\
&\geq \frac{1}{4} \int_0^1 \frac{(a(u))^2 | \partial_x u|^2}{u} \, dx - M^3 - c M(M+1),
\end{align*}
which completes the proof.
\end{proof}

From Lemma \ref{Lem:UpperBoundOfF} and Lemma \ref{Lem:LowerBoundOfF} we obtain 
$$
\frac{1}{4} \int_0^1 \frac{(a(u))^2}{u} |\partial_x u|^2 \, dx \leq {\cal F}(u) + C_2(M) \leq {\cal F}(u(t_0)) + C_1 + C_2(M).
$$
Hence
\begin{equation}\label{ineq}
\int_0^1 \frac{(a(u))^2}{u} |\partial_x u|^2 \, dx \leq C \quad \mbox{for any} \ t \in (t_0,\infty).
\end{equation}
Although we do not need it in the proof, we notice that the following inequality holds. It gives the superlinear growth of $u$. 
\begin{align*}
\int_0^1 u \int_1^u a(r) \, dr \, dx + \int_0^1 \frac{(a(u))^2}{u} |\partial_x u|^2 \, dx &= \frac{3}{2} \int_0^1 \frac{(a(u))^2}{u} |\partial_x u|^2 \, dx - {\cal F}(u) \\
&\leq \frac{3}{2} \int_0^1 \frac{(a(u))^2}{u} |\partial_x u|^2 \, dx - \frac{1}{4} \int_0^1 \frac{(a(u))^2}{u} |\partial_x u|^2 \, dx + C_2 \\
&= \frac{5}{4} \int_0^1 \frac{(a(u))^2}{u} |\partial_x u|^2 \, dx + C_2 \leq \frac{5}{4} C + C_2.
\end{align*}

We have already collected all the necessary facts for proving the main Theorem \ref{ThMain}. Below we conclude the proof. What we need is just \eqref{ineq}, mass conservation, (A0), (A1) and (A2).

\begin{proof}[Proof of Theorem \ref{ThMain}]
Observe that, by the mass conservation
\[
\int_0^1 \left|\left(\int_1^u a(s) \, ds\right)_x\right|\,dx\leq \left(\int_0^1 \frac{(a(u))^2 |\partial_x u|^2}{u} \,dx\right)^{1/2} M^{1/2},
\]
so that by \eqref{ineq} and \eqref{pr2}, we conclude that there exists $C>0$ such that
\begin{equation}\label{pr3}
\left\|\Lambda u(t, \cdot)\right\|_{W^{1,1}(0,1)}\leq C,
\end{equation}
where $\Lambda$ is given by \eqref{Lambda}. By the one-dimensional Sobolev embedding there exists $C>0$ such that 
\begin{equation}\label{pr4}
|\Lambda(u(t,\cdot))|\leq C.
\end{equation}

Notice that 
\begin{equation}\label{koniec}
|\Lambda(r)|\rightarrow \infty\;\;\mbox{if and only if}\;\;r\rightarrow \infty. 
\end{equation}
Indeed, in view of (A2), we see that $\Lambda(r)\rightarrow \infty$ for $r\rightarrow \infty$. Moreover, positivity of $a$ and (A0) imply that $\Lambda$ is bounded for $r\in[0,R]$ and any $R<\infty$. 

To complete the proof we only remark that \eqref{pr4} and \eqref{koniec} imply time-uniform boundedness of $u$.  
\end{proof}

\section{Parabolic-elliptic case}

In this section we show that our method can be applied also for the parabolic-elliptic Keller-Segel system with critical diffusion. We consider the parabolic-elliptic system of equations
\begin{eqnarray}
\label{eq-elliptic}
\left\{ \begin{array}{ll}
\partial_t u = \partial_x \left( a(u) \partial_x u - u \partial_x v \right), \ & \ \mathrm{in} \ (0,\infty) \times (0,1), \\
0 = \partial_{x}^2 v -  v + u, \  & \ \mathrm{in} \ (0,\infty) \times (0,1), \\
u(0,x) = u_0 (x), \ & \ \mathrm{for} \ x \in (0,1)
\end{array} \right.
\end{eqnarray}
under the Neumann boundary condition
$$
\partial_x u = \partial_x v = 0 \quad \quad \mathrm{on} \ (0,\infty) \times \partial (0,1) = (0,\infty) \times \{0, 1\}.
$$
We assume that the initial data $u_0 \in W^{1,\infty}(0,1)$ is a nonnegative function.
Recall that the global existence of solutions has been shown in \cite{CF2}. One also obtains in a straightforward way a claim like in Lemma \ref{oszv}. We will show the boundedness of this solution. Let us mention that such a result in a parabolic-elliptic version of Keller-Segel system, however only in its J\"{a}ger-Luckhaus type version, was found already in \cite{cl3} (see its extension in \cite{TCPhL1}). As we already mentioned in the introduction, a method used in \cite{cl3} does not seem to be extendible even to the usual parabolic-elliptic case of \eqref{eq-elliptic}. 

\begin{theo}\label{ThMain2}
Assume that (A0)-(A2) hold. The solution $(u,v)$ of \eqref{eq-elliptic} is bounded.
\end{theo}

First, we observe that Propositions \ref{lem_est_Lambda} and \ref{lem_est_uLambda} hold in the parabolic-elliptic case. To show the statement of Lemma \ref{lem_keyest} we observe that
$$
\int_0^1 a(u) \partial_x u \cdot \partial_x v \, dx = - \int_0^1 \Lambda(u) (v - u)
$$
and the remaining arguments are the same. To show Lemma \ref{lem_Keyest2} we see that in the parabolic-elliptic case
$$
\frac{d}{dt} L(u,v) = - \int_0^1 u \left| \frac{a(u)}{u} \partial_x u - \partial_x v \right|^2 \, dx.
$$

Then, instead of Lemma \ref{lemma_Lyap1}, we have the following proposition, originally proven in \cite{CF2}. 

\begin{prop}\label{prop:par-ell}
Let $(u,v)$ be a solution of \eqref{eq-elliptic} in $(0,\infty) \times (0,1)$. The following identity holds
$$
\frac{d}{dt} \mathcal{F}(u) + \mathcal{D}_1 (u,v) = \int_0^1 \frac{ua(u)v^2}{4} \, dx,
$$
where 
$$
{\cal D}_1 (u,v) := \int_0^1 ua(u) \left| \partial_x \left( \frac{a(u)}{u} \partial_x u \right) - \partial_x^2 v + \frac{v}{2} \right|^2 \, dx
$$
and $\mathcal{F}$ is given by \eqref{F}.
\end{prop}

Finally, using Proposition \ref{prop:par-ell} we are able to repeat the proof of Lemma \ref{Lem:UpperBoundOfF}, Lemma \ref{Lem:LowerBoundOfF} and the arguments from the proof of Theorem \ref{ThMain} to show Theorem \ref{ThMain2}.

\vspace{0.3cm}
\noindent
\textbf{Acknowledgement.}
Bartosz Bieganowski started working on this project during his WCMCS PhD internship at Institute of Mathematics of the Polish Academy of Sciences under supervision of Tomasz Cie\'slak, he wishes to thank for the invitation, support and warm hospitality. 
Kentarou Fujie is supported by Grant-in-Aid for Research Activity start-up (No.\,17H07131), Japan Society for the Promotion of Science. 
Takasi Senba is supported by Grant-in-Aid for Scientific Research (C) 
(No.\, 26400172), Japan Society for the Promotion of Science.

\end{document}